\documentclass[a4paper,11pt, reqno]{amsart}
\usepackage[utf8]{inputenc}
\usepackage{amssymb,amsmath}
\usepackage{amsthm}
\usepackage{amsrefs} 
\usepackage{a4wide}
\usepackage{paralist}
\usepackage{marvosym}

\usepackage{dsfont}

\newtheorem{theorem}{Theorem}[section]
\newtheorem{lemma}{Lemma}[section]

\theoremstyle{definition}

\newtheorem*{acknowledgements}{Acknowledgements}

\theoremstyle{remark}
\newtheorem{remark}{Remark}[section]

\DeclareMathOperator{\E}{\mathds{E}}
\DeclareMathOperator{\N}{\mathbb{N}}
\DeclareMathOperator{\R}{\mathbb{R}}
\DeclareMathOperator{\Defi}{\mathrel{\mathop:}=}
\DeclareMathOperator{\1}{\mathds{1}}

\title[Moderate Deviations for the REM]{On the accuracy of the normal approximation for the free energy in the REM}

\author[R. Meiners]{Raphael Meiners}
\address{Raphael Meiners: Institute for Mathematical Statistics, University of M\"{u}nster, Germany}
\email{Raphael.Meiners@uni-muenster.de}

\author[A. Reichenbachs]{Anselm Reichenbachs}
\address{Anselm Reichenbachs (\Letter): Faculty of Mathematics, Ruhr-Universit\"{a}t Bochum, Germany}
\email{Anselm.Reichenbachs@rub.de}

\date{\today}

\begin{document}

\begin{abstract}
In the present paper we consider the fluctuations of the free energy in the random energy model (REM) on a moderate deviation scale.
We find that for high temperatures the normal approximation holds only in a narrow range of scalings away from the CLT. For scalings of higher order,
probabilities of moderate deviations decay faster than exponentially.
\end{abstract}

\maketitle

\section{Introduction}
The random energy model (REM for short) is a disordered spin system from statistical mechanics, invented by Derrida in 1980 \cites{Der2,Der}. It is a toy model to describe a system of $N$ particles that can assume one of the $2^N$ accessible states from the set $\mathcal{S}_N = \{-1,+1\}^N$, called the \emph{configuration space}. The energy of a state $\sigma \in \mathcal{S}_N$ is given by $H(\sigma) = -\sqrt{N} X_{\sigma}$ where $X_{\sigma}$ is a $\mathcal{N}(0,1)$-distributed random variable, and the energies of different states are assumed to be independent, that is, $(H({\sigma}))_{\sigma\in\mathcal{S}_N}$ is (for fixed N) a sequence of i.\,i.\,d.\ normal distributed random variables. Despite its far-reaching simplifications, the REM is an important model from statistical mechanics and has been intensively studied over the last decades. More recent expositions of the model can be found in the books \cites{Bov,Tal}.

In the following, let $(\Omega,\mathcal{F},P)$ be the probability space on which the triangular array of independent $\mathcal{N}(0,1)$-distributed random variables $\big(X_{\sigma}:\sigma \in \mathcal{S}_N, N \in \N \big)$ is defined. The probability of observing a configuration $\sigma \in \mathcal{S}_N$ of the $N$ particle system is given by the random Gibbs measure \[P_{N, \beta}(\sigma) ~\Defi~ \frac{e^{-\beta H(\sigma)}}{Z_N(\beta)}\]
where $\beta>0$ is the \textit{inverse temperature} and $Z_N(\beta)$ a random normalization given by
\[Z_N(\beta) ~\Defi~ \sum_{\sigma \in \mathcal{S}_N} e^{\beta \sqrt{N} X_{\sigma}},\]
which is called \textit{partition function}. Obviously, the minus sign in the definition of the random \textit{Hamiltonian} $H(\sigma)$ and the minus sign in the definition of the Gibbs measure cancel each other, however, it is convention to use them. 

In statistical mechanics, one is interested in the existence of the so-called \textit{free energy} \[F_N(\beta) ~\Defi~ \frac{1}{N}\log Z_N(\beta)\] in the limit $N\rightarrow\infty$ in an appropriate sense. Note that this definition of the free energy differs from the one used by physicists by the factor $-\beta^{-1}$, which is constant and, therefore, omitted by mathematicians. A complete result on the existence of the free energy in the sense of almost sure convergence and convergence in $L^p$ was proved by Olivieri and Picco in 1984 \cite{OliPic} and reads as follows:

\begin{theorem}[\cite{OliPic}]
\label{LLN}
Let $\beta_c=\sqrt{2\log2}$. For all $\beta>0$
\begin{equation}\label{fe}
\lim_{N\rightarrow\infty} F_N(\beta) ~=~ F(\beta) ~\Defi~ 
\begin{cases}
\frac{\beta^2}{2}+\frac{\beta_c^2}{2} &\text{ if }\beta\leq\beta_c\\
\beta\beta_c &\text{ if }\beta > \beta_c
\end{cases}
\end{equation}
$P$-almost surely and in $L^p(\Omega,\mathcal{F},P)$ for any $1\leq p<\infty$.
\end{theorem}

\noindent The convergence in $L^1$ implies that the quenched free energy $\E F_N(\beta)$ also converges to $F(\beta)$ and, consequently, \[\lim_{N\rightarrow\infty} |F_N(\beta)- \E F_N(\beta)| ~=~ 0\] holds $P$-almost surely, which is why the free energy of the REM is said to be a self-averaging quantity. Moreover, the annealed free energy is given by
\[\frac{1}{N}\log\E Z_{N,\beta}~=~\frac{\beta^2}{2}+\frac{\beta_c^2}{2},\]
and, therefore, the quenched free energy and annealed free energy coincide in the limit $N \to \infty$ if $\beta\leq\beta_c$. This breaks down for $\beta > \beta_c$, where the quenched free energy is strictly less than the annealed free energy.

Even more, one already obtained a precise picture of free energy's deviations and fluctuations. In view of Theorem \ref{LLN}, it is a natural first step to ask for refinements of this limit theorem on the level of large deviations and, therefore, we shall briefly recall what a large deviation principle (LDP) is. For a thorough introduction to the field we refer to the books \cites{Dem,Ell}. Let $(\mathcal{X},\mathcal{B}_{\mathcal{X}})$ be a measurable space, consisting of a Hausdorff topological space $\mathcal{X}$ endowed with the Borel $\sigma$-field $\mathcal{B}_{\mathcal{X}}$. In addition to that, let $\gamma_n \to \infty$ be a sequence of real numbers and $I:\mathcal{X} \rightarrow [0,\infty]$ be a lower semicontinuous function. A sequence of random variables $(X_n)_{n \in \N}$ defined on some probability space $(S, \mathcal{A}, \mathds{P})$ with values in $(\mathcal{X},\mathcal{B}_{\mathcal{X}})$ is said to satisfy the \textit{large deviation principle} (LDP for short) with \textit{speed} $\gamma_n$ and \textit{rate function} $I$ if \[-\inf_{x\in A^{\circ}}I(x)~\leq~\liminf_{n\rightarrow\infty}\frac{1}{\gamma_n}\log \mathds{P}\left(X_n\in A\right)~\leq~\limsup_{n\rightarrow\infty}\frac{1}{\gamma_n}\log \mathds{P}\left(X_n\in A\right)~\leq~-\inf_{x\in\overline{A}}I(x)\]
for all $A\in\mathcal{B}_{\mathcal{X}}$. The rate function $I$ is said to be \textit{good} if the level sets $\left\{x \in \mathcal{X}: I(x) \leq c \right\}$ are compact subsets of $\mathcal{X}$ for all $c\in\mathbb{R}$.

As already hinted at, the probabilities of $\mathcal{O}(1)$-deviations from the limiting free energy $F(\beta)$ have already been quantified. In \cite{FedFlaMor}, Fedrigo, Flandoli and Morandin proved a large deviation theorem for the free energy, which is stated next:
 
\begin{theorem}[LDP, \cite{FedFlaMor}]
\label{LDP}
The sequence of random variables $(\frac{1}{N}\log Z_{N}(\beta))_{N \in \N}$ satisfies the LDP with speed $N$ and good rate function $I$ given by
\[I(x) ~=~ 
\begin{cases}
\infty& \text{ if }x<F(\beta)\\
0& \text{ if }x=F(\beta)\\
\frac{x^2}{2 \beta^2}-\log 2& \text{ if }x>F(\beta),
\end{cases}\]
where $F(\beta)$ are the limit points of the free energy defined in \eqref{fe}.  
\end{theorem}
\noindent Note that large deviation techniques can also be used to prove \eqref{fe} $P$-a.\,s.\ via Varadhan's Lemma (cf.\ \cite{DorWed}).

While Theorem \ref{LDP} describes the atypical behavior of $F_N(\beta)$ by studying the probabilities of large deviations, the typical behavior is described by theorems on its fluctuations, i.\,e.\ by theorems on distributional convergence of the properly rescaled free energy. This has been done by Bovier, Kurkova and L\"owe in \cite{BovKurLoe}. They proved the existence of multiple phase transitions on the level of distributional convergence and found that the fluctuations of the free energy are exponentially small. What is more, they are Gaussian if and only if $\beta \leq \sqrt{\log 2/2}$:
 
\begin{theorem}[CLT, \cite{BovKurLoe}]
\label{CLT}
\mbox{}
\begin{enumerate}[(i)]
\item
For $\beta<\sqrt{\log 2/2}$
\begin{equation*}
 e^{\frac{N}{2}(\log 2-\beta^2)}\log\left(\frac{Z_{\beta,N}}{\E Z_{\beta,N}}\right)\overset{\mathcal{D}}{\longrightarrow}\mathcal{N}(0,1).
\end{equation*}
\item
For $\beta=\sqrt{\log 2/2}$
\begin{equation*}
 e^{\frac{N}{2}(\log 2-\beta^2)}\log\left(\frac{Z_{\beta,N}}{\E Z_{\beta,N}}\right)\overset{\mathcal{D}}{\longrightarrow}\mathcal{N}(0,1/2).
\end{equation*}
\end{enumerate}
\end{theorem}

\begin{remark}
Since it will be of some importance for the present paper, we quickly want to sketch the course of action followed in \cite{BovKurLoe}: Using the Taylor expansion $\log(1+x)=x+o(x)\text{ for } x\rightarrow 0$ the authors defer the proof of a limit theorem for 
\begin{equation*}
e^{\frac{N}{2}(\log 2-\beta^2)}\log\left(\frac{Z_{\beta,N}}{\E Z_{\beta,N}}\right)~=~ e^{\frac{N}{2}(\log 2-\beta^2)}\log\left(1+\frac{Z_{\beta,N}-\E Z_{\beta,N}}{\E Z_{\beta,N}}\right)
\end{equation*}
to the more manageable random variable
\begin{equation}\label{DefofY}
e^{\frac{N}{2}(\log 2-\beta^2)} \frac{Z_{\beta,N}-\E Z_{\beta,N}}{\E Z_{\beta,N}}~=~\frac{1}{2^{N/2}}\sum_{\sigma \in \mathcal{S}_N} Y_N(\sigma)
\end{equation}
where $Y_N(\sigma)= (e^{\beta\sqrt{N}X_{\sigma}}-e^{N\beta^2/2})/e^{N\beta^2}$ are (for each $N$) i.\,i.\,d.\ random variables with mean zero and variance $s_N^2 = 1-e^{-N\beta^2} \rightarrow 1$ as $N\to \infty$. Next, the authors show that $Y_N(\sigma)$ satisfies Lindeberg's condition if $\beta < \sqrt{\log2/2}$, and obtain Theorem \ref{CLT} (i) by means of the CLT for triangular arrays. However, for $\beta = \sqrt{\log2/2} \ Y_N(\sigma)$ does \emph{not} satisfy Lindeberg's condition, which is related to the fact that $Y_N(\sigma)$'s tails become too \emph{heavy}, and the behaviour of the sum $\sum_{\sigma}Y_N(\sigma)$ is dominated by extremal events. Yet, the authors can still prove convergence in distribution to a normal distribution and attain Theorem \ref{CLT} (ii). It is worth noting that the authors also acquire complete results for $\beta > \sqrt{\log 2/2}$, where non-standard limiting distributions occur, which is due to the fact that $Y_N(\sigma)$ has even more weight on its tails in these cases.
\end{remark}

In view of Theorem \ref{CLT}, we ask the following question: can the tail probabilities \[P\left(e^{\frac{N}{2}(\log 2-\beta^2)}\log\left(\frac{Z_{\beta,N}}{\E Z_{\beta,N}}\right)>t\right)\]
be approximated by the tails of a normal distribution even for growing $t$, that is, does one find
\begin{equation*}
P\left(e^{\frac{N}{2}(\log 2-\beta^2)}\log\left(\frac{Z_{\beta,N}}{\E Z_{\beta,N}}\right)>t_N\right) ~\approx~ P(\mathcal{N}(0,1)> t_N)
\end{equation*}
even for $t_N \rightarrow \infty$? It is well-known (use e.\,g.\ \eqref{standardestimate}) that
\begin{equation*}
\lim_{N \to \infty} \frac{1}{t_N^2} \log P(\mathcal{N}(0,1)> x \,t_N) ~=~ -\frac{x^2}{2}
\end{equation*}
for any $x > 0$ and, thus, we ask for the validity of 
\begin{equation}\label{approximation}
\lim_{N \to \infty} \frac{1}{t_N^2} \log P\left(e^{\frac{N}{2}(\log 2-\beta^2)}\log\left(\frac{Z_{\beta,N}}{\E Z_{\beta,N}}\right)> x \,t_N\right) ~=~ -\frac{x^2}{2}
\end{equation}
for any $x > 0$ or, more general, for the existence of the LDP with speed $t_N^2$ and Gaussian rate function $I(x) = x^2/2$ for $\exp{(N(\log 2-\beta^2)/2)} \,t_N^{-1} \log (Z_{\beta,N}/\E Z_{\beta,N})$. Using the LDP of Theorem \ref{LDP}, we see that \eqref{approximation} does not hold if $t_N$ is of order $\Theta\big(N\,\exp{(N(\log 2-\beta^2)/2)}\big)$. Large deviation results for the remaining cases of scalings between those of the CLT and the LDP, i.\,e.\
\begin{equation*}
t_N ~\rightarrow~ \infty \text{ and } \frac{t_N}{N\,\exp{(N(\log 2-\beta^2)/2)}} ~\rightarrow~ 0,
\end{equation*}
are commonly referred to as \textit{moderate deviation results} in the literature, since one asks for deviations of $F_N(\beta)$ of order $o(1)$ from $F(\beta)$. In like manner, LDPs for scalings that are between those of the CLT and the LDP are called \emph{moderate deviation principles (MDPs)}. However, we will stick to the term LDP, since the formal definitions of the LDP and MDP are the same. Note that moderate deviations for mean field models from statistical mechanics have already been studied (cf.\ e.\,g.\ \cites{Rei,Loe}). 
\par We will show in this article that \eqref{approximation} holds if and only if $t_N = o(\sqrt{N})$, that is, \eqref{approximation} holds only in a small range of scalings close to the CLT scaling. This is particularly interesting since it is out of harmony with the general picture of moderate deviations obtained by the case of partial sums of standardized i.\,i.\,d.\ random variables $(X_i)_{i\in\mathbb{N}}$. The prototypical answer for this case is that $(t_n \sqrt{n})^{-1} \sum_{i=1}^n X_i$ satisfies under suitable conditions the LDP with speed $t_n^2$ and Gaussian rate function $I(x)=x^2/2$ for the whole range of scalings between the corresponding CLT and LLN (see \cite{EicLoe} for a necessary and sufficient condition on this type of moderate deviations). In particular, the rate function does not depend on the moderate deviation scaling.

The main result of the present paper reads as follows, where $t_N \rightarrow \infty$ is from now on a diverging sequence of real numbers:
 
\begin{theorem}[Moderate deviations for the free energy in the REM]
\mbox{}
\label{MDP}
\begin{enumerate}[(i)]
 \item
Let $\beta < \sqrt{\log2/2}$. Then, $$\frac{e^{\frac{N}{2}(\log2-\beta^2)}}{t_N} \, \log \bigg(\frac{Z_{\beta, N}}{\E Z_{\beta, N}}\bigg)$$ satisfies the large deviation principle. If $t_N = o(\sqrt{N})$, then the corresponding speed is $t_N^2$ and the good rate function is
\begin{equation*}
I(x) ~=~ \frac{x^2}{2}.
\end{equation*} 
Otherwise, if $\liminf_{n\rightarrow\infty}\frac{t_N}{\sqrt{N}}>0$, the LDP holds for any speed $\gamma_N=o(N)$ with the good rate function
\begin{equation}
I(x) ~=~ 
\begin{cases}
0&\text{ if }x=0\\
\infty&\text{ if }x \neq 0.
\end{cases}
\label{c2}
\end{equation}
\item
Let $\beta = \sqrt{\log2/2}$. Then, $$\frac{\, e^{\frac{N}{2}(\log2-\beta^2)}}{t_N} \, \log \bigg(\frac{Z_{\beta, N}}{\E Z_{\beta, N}}\bigg)$$ satisfies, for any scaling $t_N = o(\sqrt{\log N})$, the LDP with speed $t_N^2$ and good rate function $I$ given by
\begin{equation*}
I(x) ~=~ x^2.
\end{equation*}
\end{enumerate}
\end{theorem}

\begin{remark}
\mbox{}
\begin{itemize}
\item[1.] Note that the restriction $\gamma_N = o(N)$ is natural in view of the LDP (Theorem \ref{LDP}): if one considers deviations of lower order than in the LDP, then the speed of convergence to zero of these probabilities is of lower order than the speed occuring in the LDP, which was $N$ in our case. 
\item[2.] The degenerated rate function appearing in \eqref{c2} reflects the superexponential decay of moderate deviation probabilities in case of overscaling. 
\item[3.] 
Observe that for $\beta = \sqrt{\log2/2}$ the obtained rate function is $I(x) = x^2$, which matches the fact that the limiting distribution in the CLT is $\mathcal{N}(0,1/2)$ and 
\begin{equation*}
\lim_{N \to \infty} \frac{1}{t_N^2} \log P(\mathcal{N}(0,1/2)> x \,t_N) ~=~ -x^2
\end{equation*}
for any $x > 0$.
\end{itemize}
\end{remark}

\section{Proof of Theorem \ref{MDP}}
This section is devoted to the proof of Theorem \ref{MDP}, which is based on the following idea: As a first step, we follow the idea of the CLT's proof and use the approximation $\log(1+x)=x+o(x)\text{ for } x\rightarrow 0$ to defer the proof of the LDP for 
\begin{equation}\label{Var1}
\frac{e^{\frac{N}{2}(\log 2-\beta^2)}}{t_N}\,\log\left(\frac{Z_{\beta,N}}{\E Z_{\beta,N}}\right)
\end{equation}
to the proof of the LDP for 
\begin{equation}\label{Var2}
\frac{e^{\frac{N}{2}(\log 2-\beta^2)}}{t_N}\, \frac{Z_{\beta,N}-\E Z_{\beta,N}}{\E Z_{\beta,N}}.
\end{equation}
To that end, we will show in Lemma \ref{Lemma1} that the random variables \eqref{Var1} and \eqref{Var2} are exponentially equivalent (for a definition see e.\,g.\ Definition 4.2.10 in \cite{Dem}), since it is know that exponentially equivalent random variables satisfy the same LDP (see e.\,g.\ Theorem 4.2.13 in \cite{Dem}). Then, we are left to prove the LDP for the random variable
\begin{equation*}
\frac{e^{\frac{N}{2}(\log 2-\beta^2)}}{t_N}\, \frac{Z_{\beta,N}-\E Z_{\beta,N}}{\E Z_{\beta,N}} ~=~\frac{1}{t_N \, 2^{N/2}}\sum_{\sigma \in \mathcal{S}_N} Y_N(\sigma),
\end{equation*}
where $(Y_N(\sigma); \sigma \in \mathcal{S}_N, N \in \N)$ is a triangular array of independent random variables, which were defined in \eqref{DefofY}. However, the random variable $Y_N(\sigma)$ does not have finite exponential moments, which is why we use again the concept of exponential equivalence to switch over to the truncated random variables $Y^t_N(\sigma)$, where $Y^t_N(\sigma)\Defi Y_N(\sigma) \1_{\left\{Y_N(\sigma) \leq 2^{N/2}t_N^{-1}\right\}}$ (see Lemma \ref{Lemma2}), which can be studied by means of the G\"artner-Ellis theorem (cf.\ e.\,g.\ Theorem 2.3.6 in \cite{Dem}). 

We prepare the proof of Theorem \eqref{MDP} by stating and proving the above-mentioned lemmata:

\begin{lemma}\label{Lemma1}
Let $\beta \leq \sqrt{\log2/2}$. Then, 
\begin{equation*}
\frac{e^{\frac{N}{2}(\log2-\beta^2)}}{t_N} \, \log \left(\frac{Z_{\beta, N}}{\E Z_{\beta, N}}\right) \text{ and }\, 
\frac{e^{\frac{N}{2}(\log2-\beta^2)}}{t_N} \, \left(\frac{Z_{\beta, N} -\E Z_{\beta, N}}{\E Z_{\beta, N}}\right)
\end{equation*}
are exponentially equivalent for any speed $\gamma_N = o(N)$.  
\end{lemma}

\begin{proof}
Let $\varepsilon > 0$ and $T_{\beta,N} \Defi (Z_{\beta,N}-\E Z_{\beta, N})/\E Z_{\beta N}$. Since $|\log(1+x)-x| \leq x^2$ for all $x \geq -1/2$, we find
\begin{eqnarray*}
&& P\left(\left|\frac{e^{\frac{N}{2}(\log2-\beta^2)}}{t_N} \, \log \left(\frac{Z_{\beta, N}}{\E Z_{\beta, N}}\right) - \frac{e^{\frac{N}{2}(\log2-\beta^2)}}{t_N} \, \left(\frac{Z_{\beta, N} -\E Z_{\beta, N}}{\E Z_{\beta, N}}\right)\right|> \varepsilon\right) \\
&=& P\left(\left|\log \left(1+ T_{\beta, N}\right) - T_{\beta, N}\right|> \varepsilon \, t_N \,e^{-\frac{N}{2}(\log2-\beta^2)}\right)\\
&\leq& P\left(T_{\beta, N} < -\frac12\right) + P\left(T_{\beta, N}^2 > \varepsilon \, t_N\, e^{-\frac{N}{2}(\log2-\beta^2)}\right)\\
&\leq& \left(4 + \varepsilon^{-1}\, t_N^{-1}\, e^{\frac{N}{2}(\log2-\beta^2)}\right) \E T_{\beta, N}^2\\
&\leq& e^{\frac{N}{2}(\log2-\beta^2)} \E T_{\beta, N}^2
\end{eqnarray*}
for $N$ sufficiently large, where we have made use of Markov's inequality to obtain the last but one line. A direct calculation yields 
\begin{equation*}
\E T_{\beta, N}^2 ~=~ \frac{e^{N \beta^2}-1}{2^{N}} ~\leq~ e^{N(\beta^2-\log2)}
\end{equation*}
and, therefore,
\begin{eqnarray*}
&& \limsup_{N \to \infty} \frac{1}{\gamma_N} \log P\left(\left|\frac{e^{\frac{N}{2}(\log2-\beta^2)}}{t_N} \, \log \left(\frac{Z_{\beta, N}}{\E Z_{\beta, N}}\right) - \frac{e^{\frac{N}{2}(\log2-\beta^2)}}{t_N} \, \left(\frac{Z_{\beta, N} -\E Z_{\beta, N}}{\E Z_{\beta, N}}\right)\right|> \varepsilon\right)\\
&\leq& \limsup_{N \to \infty} \frac{1}{\gamma_N} \log \left(e^{\frac{N}{2}(\log2-\beta^2)}\,e^{N(\beta^2-\log2)}\right)~=~-\infty.
\end{eqnarray*}
\end{proof}

\begin{lemma}\label{Lemma2}
Let $\beta \leq \sqrt{\log2/2}$ and assume 
\begin{equation*}
t_N ~=~ 
\begin{cases}
o(\sqrt{N})&\text{ if } \beta < \sqrt{\log2/2}\\
o(\sqrt{\log N})&\text{ if } \beta = \sqrt{\log2/2}.  
\end{cases}
\end{equation*}
Then,
\begin{equation*}
\frac{1}{t_N 2^{N/2}} \sum_{\sigma \in \mathcal{S}_N} Y_N(\sigma) \text{ and }\, 
\frac{1}{t_N 2^{N/2}} \sum_{\sigma \in \mathcal{S}_N} Y_N^t(\sigma)
\end{equation*}
are exponentially equivalent on the scale $t_N^2$.  
\end{lemma}

\begin{proof}
We get
\begin{eqnarray*}
&& P\left(\left|\frac{1}{t_N 2^{N/2}} \sum_{\sigma \in \mathcal{S}_N} Y_N(\sigma) - \frac{1}{t_N 2^{N/2}} \sum_{\sigma \in \mathcal{S}_N} Y_N^t(\sigma)\right|> \varepsilon \right)\\ 
&=& P\left(\left|\frac{1}{t_N\, 2^{N/2}} \sum_{\sigma \in \mathcal{S}_N} Y_N(\sigma) \1_{\left\{Y_N(\sigma) > 2^{N/2}t_N^{-1}\right\}}\right|> \varepsilon \right)\\
&\leq& P\left(\exists \, \sigma \in \mathcal{S}_N: Y_N(\sigma) > 2^{N/2}\,t_N^{-1}\right)\\
&\leq& 2^N P\left(Y_N(\sigma_0) > 2^{N/2}\,t_N^{-1}\right)\\
&=& 2^N P\left(X_{\sigma_0} > c_N(\beta)\right)
\end{eqnarray*}
where $\sigma_0 \in \mathcal{S}_N$ and 
\begin{equation}\label{Defc}
c_N(\beta) ~\Defi~ \frac{1}{\beta \sqrt{N}} \log\left(e^{N\beta^2} 2^{N/2} t_N^{-1}+ e^{N \beta^2/2}\right) ~=~ \sqrt{N}\left(\beta+\frac{\log 2}{2\beta}\right) - \frac{\log t_N}{\beta \sqrt{N}} + o\left(N^{-1/2}\right).
\end{equation}
Making use of the standard estimate 
\begin{equation}\label{standardestimate}
\frac{x}{x^2+1}\frac{1}{\sqrt{2\pi}}\,e^{-x^2/2} ~\leq~ P(\mathcal{N}(0,1) > x) ~\leq~ \frac{1}{x} \frac{1}{\sqrt{2\pi}}\,e^{-x^2/2},
\end{equation}
which holds for all $x > 0$, we get
\begin{eqnarray*}
&& \frac{1}{t_N^2} \log P\left(\left|\frac{1}{t_N 2^{N/2}} \sum_{\sigma \in \mathcal{S}_N} Y_N(\sigma) - \frac{1}{t_N 2^{N/2}} \sum_{\sigma \in \mathcal{S}_N} Y_N^t(\sigma)\right|> \varepsilon \right)\\ 
&\leq& \frac{1}{t_N^2} \log \left(2^N \frac{1}{c_N(\beta) \sqrt{2\pi}}\,e^{-c_N(\beta)^2/2} \right)\\
&=& \frac{1}{t_N^2} \log \left(2^N \frac{1}{\sqrt{N}}\,e^{-c_N(\beta)^2/2} \right) +o(1) \\
&=& \frac{N}{t_N^2} \log2 - \frac{c_N(\beta)^2}{2 t_N^2} - \frac{\log N}{2 t_N^2} +o(1) \\
&=& - \frac{N}{2 t_N^2} \left(\beta - \frac{\log2}{2\beta}\right)^2 - \frac{\log N}{2 t_N^2} +o(1)\rightarrow -\infty
\end{eqnarray*}
as $N \to \infty$. Note that $(\beta - \log2/(2\beta))^2 > 0$ if and only if $\beta \neq \sqrt{\log2/2}$ so that the last line follows from the conditions made on the asympotic behavior of $t_N$.
\end{proof}

Now that we have gathered all preliminary results, we can start with a proof of this article's main theorem:

\begin{proof}[Proof of Theorem \ref{MDP}]
We start with a proof of (i)'s first part and (ii). To that purpose, let $\beta \leq \sqrt{\log 2/2}$ and assume 
\begin{equation*}
t_N ~=~ 
\begin{cases}
o(\sqrt{N})&\text{ if } \beta < \sqrt{\log2/2}\\
o(\sqrt{\log N})&\text{ if } \beta = \sqrt{\log2/2}.  
\end{cases}
\end{equation*}
By means of Lemma \ref{Lemma1} and Lemma \ref{Lemma2} it suffices to prove the desired LDP for 
\begin{equation*}
\frac{1}{t_N 2^{N/2}} \sum_{\sigma \in \mathcal{S}_N} Y_N^t(\sigma).
\end{equation*}
This follows directly from the G\"artner-Ellis theorem once we have proved
\begin{equation*}
\lim_{N \to \infty} \frac{1}{t_N^2} \log \E \left[e^{\lambda \,t_N^2 \,\frac{1}{t_N \, 2^{N/2}}\,\sum_{\sigma \in \mathcal{S}_N} Y_N^t(\sigma)}\right] ~=~ \Lambda(\lambda) ~\Defi~ \begin{cases}
\frac{\lambda^2}{2}& \text{ if } \beta < \sqrt{\frac{\log 2}{2}}\\
\frac{\lambda^2}{4}& \text{ if } \beta = \sqrt{\frac{\log 2}{2}}
\end{cases} 
\end{equation*}
for all $\lambda \in \R$. Since
\begin{equation*}
t_N^{-2} \log \E \left[e^{\lambda \,t_N^2 \,\frac{1}{t_N \, 2^{N/2}}\,\sum_{\sigma \in \mathcal{S}_N} Y_N^t(\sigma)}\right]
~=~ t_N^{-2}\, 2^N \log \left(1+\left(\E \left[e^{\lambda \,t_N \,2^{-N/2} Y_N^t(\sigma_0)}\right]-1\right)\right) 
\end{equation*}
for any $\sigma_0 \in \mathcal{S}_N$, this follows, using the Taylor expansion $\log(1+x) = x + \mathcal{O}(x^2)$ as $x \rightarrow 0$, from
\begin{equation}\label{main}
t_N^{-2}\, 2^N\left(\E \left[e^{\lambda \,t_N \,2^{-N/2} Y_N^t(\sigma_0)}\right] -1\right) ~=~ \Lambda(\lambda) + o(1), 
\end{equation}
which we are going to prove in the sequel. To that purpose, we calculate the asymptotics of the first three moments of $Y^t_N(\sigma_0)$ and get 
\begin{eqnarray}
\E Y^t_N(\sigma_0) &=& o\big(t_N \,2^{-N/2}\big), \label{firstmoment} \\
\E Y^t_N(\sigma_0)^2 &=& 2 \, \Lambda(\lambda) +o(1), \label{secondmoment}\\
\E |Y^t_N(\sigma_0)|^3 &=& o(t_N^{-1}\, 2^{N/2}) \label{thirdmoment}. 
\end{eqnarray} 

Ad \eqref{firstmoment}: With $c_N(\beta)$ (cf.\ \eqref{Defc}) we have
\begin{eqnarray*} 
\E Y^t_N(\sigma_0)
&=& e^{-N \beta^2} \E \left[\left(e^{\sqrt{N} \beta X_{\sigma_0}} - e^{N \beta^2/2}\right)\1_{\{e^{\sqrt{N}\beta X_{\sigma_0}}-e^{N\beta^2/2} \leq \,2^{N/2}\, e^{N \beta^2}\,t_N^{-1}\}}\right]\\
&=& e^{-N \beta^2/2} \left(\frac{1}{\sqrt{2\pi}} \int_{-\infty}^{c_N(\beta)} e^{-\frac12 (x-\sqrt{N}\beta)^2}dx - P\big(X_{\sigma_0} \leq c_N(\beta)\big)\right)\\
&=& e^{-N \beta^2/2} P\big(X_{\sigma_0} > c_N(\beta)-\sqrt{N}\beta\big) \left(\frac{P\left(X_{\sigma_0} > c_N(\beta) \right)}{P\big(X_{\sigma_0} > c_N(\beta)-\sqrt{N}\beta\big)}-1\right).
\end{eqnarray*}
Using the standard estimate \eqref{standardestimate} for a Gaussian random variable, we see 
\begin{equation*}
\frac{P\left(X_{\sigma_0} > c_N(\beta)\right)}{P\big(X_{\sigma_0} > c_N(\beta)-\sqrt{N}\beta\big)} ~=~ o(1)
\end{equation*}
and
\begin{equation*}
P\big(X_{\sigma_0} > c_N(\beta)-\sqrt{N}\beta\big) ~=~ o\left(e^{-(c_N(\beta)-\sqrt{N}\beta)^2/2}\right),
\end{equation*}
which yields \eqref{firstmoment} as
\begin{eqnarray}\label{eq} 
t_N^{-1} \, 2^{N/2} \E Y^t_N(\sigma_0)
&=& o\left(t_N^{-1} \, e^{N(\log2/2 - \beta^2/2)} \,e^{-\frac{N}{2} \left(\frac{\log2}{2\beta}\right)^2+\frac{\log2 \log t_N}{2\beta^2}}\right)\nonumber\\
&=& o\left((t_N)^{\log2/(2\beta^2)-1} \, e^{-\frac{N}{2} \left(\beta -\frac{\log2}{2\beta}\right)^2}\right)\nonumber\\
&=& o(1),
\end{eqnarray}
where we have used in the last line that 
\begin{itemize}
\item[$\cdot$] $\log2/(2\beta^2)-1 = 0$ and $(\beta-\log2/(2\beta))^2 = 0$ if $\beta = \sqrt{\log2/2}$ and
\item[$\cdot$] $(\beta-\log2/(2\beta))^2 > 0$ if $\beta < \sqrt{\log2/2}$.
\end{itemize}

Ad \eqref{secondmoment}: It is
\begin{eqnarray*}
\E Y_N^t(\sigma_0)^2 
&=& \frac{1}{\sqrt{2\pi}} \int_{-\infty}^{c_N(\beta)} e^{-\frac12 x^2} \left(\frac{e^{\sqrt{N}\beta x}-e^{N\beta^2/2}}{e^{N\beta^2}}\right)^2dx\\
&=& \frac{1}{\sqrt{2\pi}} \int_{-\infty}^{c_N(\beta)} e^{-\frac12 x^2+2\sqrt{N}\beta x-2N \beta^2}dx +o(1)\\
&=& \frac{1}{\sqrt{2\pi}} \int_{-\infty}^{\sqrt{N}(\log2/(2\beta)-\beta)+o(1)} e^{-\frac12 x^2}dx +o(1)\\
&\rightarrow& 
\begin{cases} 
1& \text{ if } \beta < \sqrt{\frac{\log 2}{2}}\\
\frac12& \text{ if } \beta = \sqrt{\frac{\log 2}{2}}
\end{cases}
\text{ as } N \to \infty.
\end{eqnarray*}

Ad \eqref{thirdmoment}: For every $\varepsilon > 0$ it is 
\begin{eqnarray*}
&& t_N \,2^{-\frac{N}{2}} \E |Y_N^t(\sigma_0)|^3 \\
&=& t_N \,2^{-\frac{N}{2}} \E \left[|Y_N^t(\sigma_0)|^3\1_{\{|Y_N(\sigma_0)|\leq \varepsilon\,t_N^{-1} \,2^{N/2}\}}\right] + t_N \,2^{-\frac{N}{2}} \E\left[|Y_N^t(\sigma_0)|^3\1_{\{|Y_N(\sigma_0)|> \varepsilon \,t_N^{-1}\,2^{N/2}\}}\right]\\
&\leq& \varepsilon \E Y_N^t(\sigma_0)^2 + t_N^{-2}\, 2^N P\big(|Y_N(\sigma_0)| > \varepsilon\,t_N^{-1}\, 2^{N/2}\big)\\
&=& \varepsilon \E Y_N^t(\sigma_0)^2 + t_N^{-2} \,2^N P\big(X_N(\sigma_0) > c_N(\beta) + \mathcal{O}(N^{-1/2})\big)\\
&=& \varepsilon \E Y_N^t(\sigma_0)^2 + o\big(t_N^{-2} \,2^N e^{-c_N(\beta)^2/2}\big)\\
&=& \varepsilon \E Y_N^t(\sigma_0)^2 + o\big(t_N^{-2} \,e^{N \log2 - \frac{N}{2}(\beta + \log2/(2\beta))^2 + (1+\log2/(2\beta))\log t_N}\big)\\
&=& \varepsilon \E Y_N^t(\sigma_0)^2 + o\big(t_N^{\log2/(2\beta^2)-1}e^{-\frac{N}{2}(\beta-\log2/(2\beta))^2}\big)\\
&=& \varepsilon \E Y_N^t(\sigma_0)^2 + o(1),
\end{eqnarray*} 
where we have used the same argument as in \eqref{eq} to derive the last line. Thus, with the help of \eqref{secondmoment} we see
\begin{equation*}
\lim_{N \to \infty} t_N \,2^{-N/2} \E |Y_N^t(\sigma_0)|^3 ~\leq~ \varepsilon
\end{equation*}   
which yields \eqref{thirdmoment} as $\varepsilon$ was arbitrary. 

Now, we see that \eqref{main} follows with the help of \eqref{firstmoment} and \eqref{secondmoment} from
\begin{equation*}
\E \left[e^{\lambda \,t_N \,2^{-N/2} Y_N^t(\sigma_0)} - \sum_{i=0}^2 \frac{\left(\lambda \,t_N \,2^{-N/2} Y_N^t(\sigma_0)\right)^i}{i!}\right] ~=~ o\left(\frac{t_N^2}{2^N}\right). 
\end{equation*}
Since $\lambda \,t_N \,2^{-N/2} Y_N^t(\sigma_0)$ is bounded by $\lambda$ it can easily be seen, using the Lagrange form of the remainder in Taylor's formula, that
\begin{eqnarray*}
&& \left|\E \left[e^{\lambda \,t_N \,2^{-N/2} Y_N^t(\sigma_0)} - \sum_{i=0}^2 \frac{\left(\lambda \,t_N \,2^{-N/2} Y_N^t(\sigma_0)\right)^i}{i!}\right]\right| ~\leq~ \frac{e^{\lambda}}{3!} \lambda^3 \,t_N^3 \,2^{-3 N/2} \E \left[\mid Y_N^t(\sigma_0)\mid^3\right],
\end{eqnarray*}
which finishes the proofs of (i)'s first part and (ii) with the help of \eqref{thirdmoment}.

For the the second part of (i), let $\gamma_N = o(N)$ be an arbitrary speed. It suffices to prove
\begin{eqnarray}
\lim_{N \to \infty}\frac{1}{\gamma_N} \log P\left(\left|\frac{e^{\frac{N}{2}(\log2-\beta^2)}}{t_N} \, \log \bigg(\frac{Z_{\beta, N}}{\E Z_{\beta, N}}\bigg)\right|> \varepsilon\right) ~=~ - \infty,\label{eq1}\\
\lim_{N \to \infty}\frac{1}{\gamma_N} \log P\left(\left|\frac{e^{\frac{N}{2}(\log2-\beta^2)}}{t_N} \, \log \bigg(\frac{Z_{\beta, N}}{\E Z_{\beta, N}}\bigg)\right|\leq \varepsilon\right) ~=~ 0\label{eq2}
\end{eqnarray}
for any $\varepsilon > 0$. The validity of \eqref{eq1} follows directly from
\begin{eqnarray*}
&&\frac{1}{\gamma_N} \log P\left(\left|\frac{e^{\frac{N}{2}(\log2-\beta^2)}}{t_N} \, \log \bigg(\frac{Z_{\beta, N}}{\E Z_{\beta, N}}\bigg)\right|> \varepsilon\right)\\
&=& \frac{1}{\gamma_N} \log P\left(\left|\frac{e^{\frac{N}{2}(\log2-\beta^2)}}{\sqrt{\gamma_N}} \, \log \bigg(\frac{Z_{\beta, N}}{\E Z_{\beta, N}}\bigg)\right|> \varepsilon\frac{t_N}{\sqrt{\gamma_N}}\right)
\end{eqnarray*}
since (it holds $\liminf t_N/\sqrt{N} > 0$ and $\gamma_N = o(N)$ in this case)
\begin{equation*}
\frac{t_N}{\sqrt{\gamma_N}} ~=~ \frac{t_N}{\sqrt{N}} \sqrt{\frac{N}{\gamma_N}} ~\rightarrow~ \infty
\end{equation*} 
and
\begin{equation*}
\lim_{N \to \infty} \frac{1}{\gamma_N} \log P\left(\left|\frac{e^{\frac{N}{2}(\log2-\beta^2)}}{\sqrt{\gamma_N}} \, \log \bigg(\frac{Z_{\beta, N}}{\E Z_{\beta, N}}\bigg)\right|> \delta\right) ~=~ -\frac{\delta^2}{2} 
\end{equation*}
for any $\delta > 0$ by the first part of (i), which we proved above. Finally, this also yields the validity of \eqref{eq2} as \eqref{eq1} implies
\begin{equation*}
\lim_{N \to \infty} P\left(\left|\frac{e^{\frac{N}{2}(\log2-\beta^2)}}{t_N} \, \log \bigg(\frac{Z_{\beta, N}}{\E Z_{\beta, N}}\bigg)\right|\leq \varepsilon\right) ~=~ 1.
\end{equation*}
\end{proof}

\begin{remark} 
A LDP for $$\frac{e^{\frac{N}{2}(\log2-\beta^2)}}{t_N} \, \log \bigg(\frac{Z_{\beta, N}}{\E Z_{\beta, N}}\bigg)$$ in the case $\beta = \sqrt{\log 2/2}, \liminf_{N \to \infty} t_N/\sqrt{\log N} > 0$ is still an open question. By Lemma \ref{Lemma1} this random variable is exponentially equivalent to $t_N^{-1} 2^{-N/2} \sum_{\sigma \in \mathcal{S}_N} Y_N(\sigma)$ and it can even be shown that $t_N^{-1} 2^{-N/2} \sum_{\sigma \in \mathcal{S}_N} Y_N^t(\sigma)$ satisfies the LDP with speed $t_N^2$ and rate function $I(x) = x^2$ under the natural condition $t_N = o(\sqrt{N})$. However, one can show that in this case $t_N^{-1} 2^{-N/2} \sum_{\sigma \in \mathcal{S}_N} Y_N(\sigma)$ and $t_N^{-1} 2^{-N/2} \sum_{\sigma \in \mathcal{S}_N} Y_N(\sigma)$ are \emph{not} exponentially equivalent, since $Y_N(\sigma)$'s tails become too heavy and extremal events start to dominate the sum's behavior. This is the same effect that can be observed in the CLT, where it engenders a breakdown of the standard CLT. 
\end{remark} 

\begin{acknowledgements}
The author R.\ Meiners is supported by the German National Academic Foundation and the author A.\ Reichenbachs by Deutsche Forschungsgemeinschaft via SFB \textbar TR12. The authors also thank Peter Eichelsbacher and Matthias L\"{o}we for suggesting this project and fruitful discussions.
\end{acknowledgements}

\bibliography{REMpaperbib}

\end{document}